\documentclass{amsart}
\usepackage{amsfonts,xcolor}
\usepackage{amsthm}
\usepackage{latexsym}
\usepackage{amsmath}
\usepackage{amssymb}
\usepackage{arydshln}
\newtheorem{thm}{Theorem}
\theoremstyle{plain}

\newtheorem{coro}{Corollary}

\newtheorem{defi}{Definition}
\newtheorem{exam}{Example}

\newtheorem{lem}{Lemma}

\newtheorem{prop}{Proposition}
\newtheorem{remark}{Remark}

\numberwithin{equation}{section}
\usepackage{geometry}

\begin{document}
\title{On the regularity of the Hankel determinant sequence of the characteristic sequence of powers}
\author{Yingjun Guo}
%\address[A. One and A. Two]{Author OneTwo address line 1\\
%Author OneTwo address line 2}
%\email[A. One]{aone@aoneinst.edu}
%\urladdr{http://www.authorone.oneuniv.edu}

\begin{abstract}
For any sequences $\mathbf{u}=\{u(n)\}_{n\geq0}, \mathbf{v}=\{v(n)\}_{n\geq0},$ we define $\mathbf{u}\mathbf{v}:=\{u(n)v(n)\}_{n\geq0}$ and $\mathbf{u}+\mathbf{v}:=\{u(n)+v(n)\}_{n\geq0}$.
Let $f_i(x)~(0\leq i< k)$ be sequence polynomials whose coefficients are integer sequences. We say an integer sequence $\mathbf{u}=\{u(n)\}_{n\geq0}$ is a polynomial generated sequence if
$$\{u(kn+i)\}_{n\geq0}=f_i(\mathbf{u}),~(0\leq i< k).$$
%Here we define $\mathbf{u}\mathbf{v}:=\{u(n)v(n)\}_{n\geq0}$ and $\mathbf{u}+\mathbf{v}:=\{u(n)+v(n)\}_{n\geq0}$ for any two sequences $\mathbf{u}=\{u(n)\}_{n\geq0}, \mathbf{v}=\{v(n)\}_{n\geq0}.$

In this paper, we study the polynomial generated sequences.  Assume $k\geq2$ and $f_i(x)=\mathbf{a}_ix+\mathbf{b}_i~(0\leq i< k)$. If $\mathbf{a}_i$ are $k$-automatic and  $\mathbf{b}_i$ are $k$-regular for $0\leq i< k$, then we prove that the  corresponding polynomial generated sequences are $k$-regular.  As a application, we prove that  the Hankel determinant sequence $\{\det(p(i+j))_{i,j=0}^{n-1}\}_{n\geq0}$ is $2$-regular, where $\{p(n)\}_{n\geq0}=0110100010000\cdots$ is the  characteristic sequence of powers 2.  Moreover, we give a answer of Cigler's conjecture about the Hankel determinants.\end{abstract}
%\begin{keyword}
%Paperfolding sequence \sep   Hankel determinant \sep Irrationality exponent \sep  Automatic sequence
%\end{keyword}

\maketitle

%\address[A. One and A. Two]{Author OneTwo address line 1\\
%Author OneTwo address line 2}
%\email[A. One]{aone@aoneinst.edu}
%\urladdr{http://www.authorone.oneuniv.edu}

%\tableofcontents

%%%%%%%%%%%%%%%%%%%%%%%%%%%%%%%%%%%%%%%%%%%%%%%%%%%%%%%%%%%%%%%%%%%%%%%%%%%%%%%%%%
\section{Introduction}
To introduce our motivation for the problem in this paper, we first recall some basic definitions of automatic and regular sequences.

%@@@@@@@@@@@@@@@@@@@@@@@@@@@@@@@@@@@@@@@@@@@@@@@@@@@@@@@@@@@@@@@@@@@@@
\subsection{Automatic and regular sequences}
We say a sequence $\mathbf{u} = \{u(n)\}_{n\geq0}$ with values in a finite set \emph{$k$-automatic} if,  informally speaking, $u(n)$ is a finite-state function of the base-$k$ expansion of $n$ \cite{Cob,AS03}.  This is equivalent to the fact that the $k$-kernel $\mathcal{K}_k(\mathbf{u})$ is a finite set  \cite{CKM,E1974}, where the  \emph{$k$-kernel}  is a collection of subsequences  $$\mathcal{K}_k(\mathbf{u})=\Big\{\{u(k^in+j)\}_{n\geq0}:i\geq0,0\leq j<k^{i}\Big\}.$$

While all $k$-automatic sequences are defined over finite alphabets, Allouche and Shallit \cite{AS1992, jp} introduced a wider class of $k$-regular sequences that are allowed to take values in a Noetherian ring $R$. A sequence is \emph{$k$-regular} if the module generated by its $k$-kernel is finitely generated. In this paper, unless otherwise stated, the sequences we considered are integer sequences and assume the underlying ring is $\mathbb{Z}$. More precisely, we say that an integer sequence $\{u(n)\}_{n\geq0}$ is \emph{$k$-regular} if every sequence of its $k$-kernel is a $\mathbb{Z}$-linear combination of a finite set. That is to say, there exist a finite number of integer sequences $\{a_1(n)\}_{n\geq0}, \{a_2(n)\}_{n\geq0},\cdots,\{a_N(n)\}_{n\geq0}$ such that for any $i\geq0,0\leq j<k^i$, there exist $c_1,c_2,\cdots,c_N\in\mathbb{Z}$ such that
$$u(k^in+j)=\sum_{\ell=1}^Nc_{\ell}a_{\ell}(n),~(n\geq0).$$

The $k$-regular sequences play the same role for integer valued sequences as the $k$-automatic sequences play for sequences over a finite alphabet.
More relations between the $k$-regular sequences and the $k$-automatic sequences can be found in \cite{AS1992, jp,CRS}.
%In \cite{AS1992, jp},  Allouche and Shallit proved that the regular sequences forms a ring.
\medskip

%@@@@@@@@@@@@@@@@@@@@@@@@@@@@@@@@@@@@@@@@@@@@@@@@@@@@@@@@@@@@@@@@@@@@@

\subsection{Polynomial generated sequences}
For any  sequences $\mathbf{u}=\{u(n)\}_{n\geq0}$ and $\mathbf{v}=\{v(n)\}_{n\geq0}$.  We define addition and multiplication as follows:
\begin{itemize}
  \item $\mathbf{u}+\mathbf{v}:=\{u(n)+v(n)\}_{n\geq0},$
  \item $\mathbf{u}\cdot\mathbf{v}:=\{u(n)v(n)\}_{n\geq0}.$
\end{itemize}
Let $\mathcal{R}$ denote the set of all integer sequences.  Then $(\mathcal{R},+,\cdot)$ forms a commutative ring.  Similarly, $\mathcal{R}[X]$, the set of polynomials in the indeterminate $x$ over $\mathcal{R}$, is the set of all expressions of the form
$$\mathbf{a}_0+\mathbf{a}_1 x+\cdots+\mathbf{a}_m x^m.$$
%together with the operations of addition and multiplication as usually defined.  It is easy to see that $\mathcal{R}[X]$ is a commutative ring with unit element $\mathbf{1}$.
Each element of $\mathcal{R}[X]$ is call a \emph{sequence polynomial}.  If $\mathbf{a}_m$ is a nonzero sequence, then $m$ is called the \emph{degree} of $f$.
%If $\mathbf{a}_i=\{a_i(n)\}_{n\geq0}$ are $k$-automatic sequences, then the sequence polynomial is called to be  a \emph{$k$-automatic sequence polynomial}. In particular,
If $\mathbf{a}_i=\{a_i\}_{n\geq0}$ are constant sequences,  then the sequence polynomial is called to be  a \emph{constant sequence polynomial}, and denoted briefly by $a_0+a_1 x+\cdots+a_n x^n$.
%It is easy to see that $\mathcal{R}[X]$ is a commutative ring with unit element $\mathbf{1}$.

\smallskip

A sequence polynomial can be considered as a map from $\mathcal{R}$ to $\mathcal{R}$.
Let $\mathbf{u}$ be an integer sequence and $f(x)=\mathbf{a}_0+\mathbf{a}_1 x+\cdots+\mathbf{a}_m x^m$ be a sequence polynomial with degree $m$ in $\mathcal{R}[X]$. Then the \emph{image} of $\mathbf{u}$ under the map $f(x)$ is the integer sequence $f(\mathbf{u})=\mathbf{a}_0+\mathbf{a}_1\mathbf{u}+\cdots+\mathbf{a}_m \mathbf{u}^m$.
In particular, if $f(x)=a_0+a_1 x$, then $f(\mathbf{u})$ is called to be a \emph{linear polynomial} of $\mathbf{u}$.

\smallskip

\begin{defi}\label{poly1}
Given an integer sequence $\mathbf{u}=\{u(n)\}_{n\geq0}$. If there exists an integer $k\geq1$ and sequence polynomials $f_i(x)\in\mathcal{R}[X]~(0\leq i<k)$ such that
$$\{u(kn+i)\}_{n\geq0}=f_i(\mathbf{u}),$$
then we say that $\mathbf{u}$ is a \emph{polynomial generated sequence.} The set of polynomials $\{f_i(x)\in\mathcal{R}[X]:0\leq i< k\}$ is called to be a \emph{generated polynomial system}. The set of the sequences generated by polynomials $\{f_i(x)\in\mathcal{R}[X]:0\leq i< k\}$ is denoted by $\mathcal{G}(f_0,f_1,\cdots,f_{k-1})$.
\end{defi}
\smallskip
Assume $f_i(x)=\mathbf{a}_ix+\mathbf{b}_i$, where $\mathbf{a}_i=\{a_i(n)\}_{n\geq0}, \mathbf{b}_i=\{b_i(n)\}_{n\geq0}$ for $0\leq i< k.$  If $\mathbf{u}=\{u(n)\}_{n\geq0}\in\mathcal{G}(f_0,f_1,\cdots,f_{k-1})$, then, for $0\leq i< k, n\geq0,$
$$u(kn+i)=a_i(n)u(n)+b_i(n).$$
Hence, we have
\begin{eqnarray*}
u(0)& = & \frac{b_0(0)}{1-a_0(0)},\\
u(i) & = & a_i(0)u(0)+ b_i(0) ~\text{for $1\leq i< k$},\\
\cdots &\cdots&\cdots\cdots
\end{eqnarray*}
Note from above that $u(n)$ are determined by $a_i(n)$ and $b_i(n)$ for $0\leq i< k, n\geq0$. Moreover, if $u(0)$ is integer, then $\mathbf{u}$ is an integer sequence.

Hence, if the sequence polynomials $f_i(x)\in\mathcal{R}[X]~(0\leq i<k)$ are degree $1$, we always assume $b_0(0)=N(1-a_0(0))$ for some integer $N$, and we define $u(0)=0$ if $a_0(0)=1$ in this paper.  In this case, the set $\mathcal{G}(f_0,f_1,\cdots,f_{k-1})$ always exists and has exactly one sequence $\mathbf{u}=\{u(n)\}_{n\geq0}$, and we denote the polynomial generated sequence $\mathbf{u}$ by $\mathcal{G}(f_0,f_1,\cdots,f_{k-1})$ briefly.

\begin{exam}\label{coro-poly1}
\begin{enumerate}
\item The polynomial generated sequence $\mathcal{G}(x,-x+1)=\{t(n)\}_{n\geq0}$ is the famous Thue-Morse sequence.
%\item The polynomial generated sequence $\mathcal{G}(1,-x+1)=\{d(n)\}_{n\geq0}$ is the period-doubling sequence.
  \item Assume $k\geq2$. The polynomial generated sequence $$\mathcal{G}(x,x+1\cdots,x+k-1)=\{s_k(n)\}_{n\geq0}$$ is a $k$-regular sequence,
where $s_k(n)$ is the sum of the digits in the base-$k$ representation of $n$. \end{enumerate}
\end{exam}

So, given a generated polynomial system $\{f_i(x)\in\mathcal{R}[X]:0\leq i< k\}$, where the polynomials are degree $1$, we can obtain a set of integer sequences $\mathcal{G}(f_0,f_1,\cdots,f_{k-1})$. Then, a problem appears. What can be said about properties of $\mathcal{G}(f_0,f_1,\cdots,f_{k-1})$?

If $f_i(x)$ are linear polynomials for $0\leq i< k$ with $k\geq2$, then it is easy to check that  $\mathcal{G}(f_0,f_1,\cdots,f_{k-1})$ is a $k$-regular sequence. Moreover, if $f_i(x)=a_ix+\mathbf{b}_i$ for $0\leq i< k$, where $a_i$ are integers and $\mathbf{b}_i$ are $k$-regular, then $\mathcal{G}(f_0,f_1,\cdots,f_{k-1})$ is also $k$-regular.
The following theorem gives a general result.
\smallskip
\begin{thm}\label{polynomial1}
Assume $k\geq2$ and $f_i(x)=\mathbf{a}_ix+\mathbf{b}_i$ for $0\leq i< k$. If $\mathbf{a}_i$ are $k$-automatic and $\mathbf{b}_i$ are $k$-regular, then the polynomial generated sequence $\mathcal{G}(f_0,f_1,\cdots,f_{k-1})$ is $k$-regular.
\end{thm}
\smallskip
%In particular,  we have the following corollary.
%\begin{coro}\label{poly}
%Assume $k\geq2$ and $f_i(x)=a_ix+b_i$, where $a_i, b_i$ are integers for $0\leq i< k$, then the polynomial generated sequence $\mathcal{G}(f_0,f_1,\cdots,f_{k-1})$ is $k$-regular.
%\end{coro}

\medskip

Let $\mathbf{u}=\{u(n)\}_{n\geq0}$ be a sequence, then we define the \emph{shift map} $S(\mathbf{u})$ to be the sequence $\{u(n+1)\}_{n\geq0}$. Similarly,  we have $S^k(\mathbf{u}) = u(k)u(k+1)u(k+2)\cdots$ for $k\geq0$. Let $f=\mathbf{a}_0+\mathbf{a}_1 x+\cdots+\mathbf{a}_m x^m$ be a sequence polynomial. Then, we define the \emph{composition} $f\circ S(\mathbf{u})$, to be the sequence $f(S(\mathbf{u}))$.

\smallskip
%Let $f_i(x)=\mathbf{a_i}x+\mathbf{b_i}$ be sequence polynomials and $S_i\in\{S^j:j\geq0\}$ be shifts for $0\leq i<k$.
\begin{defi}\label{defi2}
Given an integer sequence $\mathbf{u}=\{u(n)\}_{n\geq0}$. If there exist integers $k\geq1, N\geq0$, polynomials $f_i(x)\in\mathcal{R}[X]$ and shifts $S_i\in\{S^j:j\geq0\}$ for $0\leq i<k$ such that
$$u(kn+i)=f_i\circ S_i(\mathbf{u})(n), ~(n\geq N)$$
then we say that $\mathbf{u}$ is a \emph{polynomial generated sequence with shift}.
\end{defi}
\smallskip

The following theorem tells us that the shift do not change the regularity of  the polynomial generated sequence.
%Always, we assume $S_0=S^0$
%Let $\mathbf{u}=\{u(n)\}_{n\geq0}$ be a sequence defined by
%\begin{equation}
%\label{shift}
%\{u(kn+i)\}_{n\geq0}=f_i\circ S_i(\mathbf{u}),~(0\leq i<k)
%\end{equation}
%where the \emph{shift} $S^i(\mathbf{u})$ is defined by the sequence $\{u(n+i)\}_{n\geq0}$ for $i\geq0$.  Hence, for $n\geq0$, we have
%$$u(2n)=a_0(n)u(n)+b_0(n)~\text{and}~u(2n+1)=a_1(n)u(n+1)+b_1(n).$$
%Then, we have the following conclusion.
%\smallskip
\begin{thm}\label{polynomial2}
Assume $k\geq2$ and $f_i(x)=\mathbf{a}_ix+\mathbf{b}_i$ for $0\leq i<k$. If $\mathbf{a}_i$ are $k$-automatic and $\mathbf{b}_i$ are $k$-regular, then the polynomial generated sequences with shift  are $k$-regular.
\end{thm}
In particular,  we have
\begin{coro}\label{coro-poly2}
 Let $\{a_i(n)\}_{n\geq0}$ and $\{b_i(n)\}_{n\geq0}$ be $2$-automatic sequences for $0\leq i< 2$. Assume $\mathbf{u}=\{u(n)\}_{n\geq0}$ is a sequence defined by
$$u(2n)=a_0(n)u(n)+b_0(n),~~u(2n+1)=a_1(n)u(n+1)+b_1(n).$$
Then the sequence $\mathbf{u}$ is $2$-regular.
\end{coro}

%Let $f_0(x)=\mathbf{a_0}x+\mathbf{b_0}$ and $f_1(x)=\mathbf{a_1}x+\mathbf{b_1}$ be two sequence polynomials. If

\medskip
%@@@@@@@@@@@@@@@@@@@@@@@@@@@@@@@@@@@@@@@@@@@@@@@@@@@@@@@@@@@@@@@@@@@@@

\subsection{Hankel determinants of the characteristic sequence of the power of $2$}

Let $\mathbf{u}=\{u(n)\}_{n\geq0}$ be  a sequence of real numbers. For every integer $k\geq0$, define a  Hankel matrix $\mathbf{u}_{m,n}^{k}$ of order $m\times n$ associated with $\mathbf{u}$ as follows:
$$\mathbf{u}_{m,n}^{k}=\begin{pmatrix}
    u(k)  &   u(k+1) & \cdots &  u(k+n-1) \\
    u(k+1)  &   u(k+2) & \cdots &  u(k+n) \\
    \vdots &   \vdots & \ddots &  \vdots \\
    u(k+m-1)  &   u(k+m) & \cdots &  u(k+m+n-2) \\
\end{pmatrix}.$$

Note that the rows of $\mathbf{u}_{m,n}^{k}$ are made up of successive length-$n$ ``windows" into the sequence  $\mathbf{u}$.  If $m=n$, we always use the symbols $\mathbf{u}_{n}^{k}$ and $|\mathbf{u}_{n}^{k}|$ to stand for the $n$-order Hankel matrix and $n$-order Hankel determinant respectively.

Hankel determinants associated with a sequence play an important role in the study of the moment problem, Pad\'{e} approximation, and the combinatorial properties of sequence \cite{BG,BHWY,GWW,L02,KTW}. Given a $k$-automatic integer sequence $\mathbf{u}=\{u(n)\}_{n\geq0}$, we obtain a sequence of Hankel determinants $|\mathbf{u}_{n}^{m}|$.  Note that the Hankel determinants $|\mathbf{u}_{n}^{m}|$ are determined by the block $u(m)u(m+1)\cdots u(m+2n)$ for any fixed $n\geq1$. And the block sequence $\{u(m)u(m+1)\cdots u(m+2n)\}_{m\geq0}$ is $k$-automatic. Hence,  the determinant sequence $\{|\mathbf{u}_{n}^{m}|\}_{m\geq0}$ is $k$-automatic, please see \cite{APWW98}.

There are some results about the automaticity of the Hankel determinant sequences. Allouche, Peyri\`ere, Wen and Wen first studied the Hankel determinant of the Thue-Morse sequence $\mathbf{t}$ in \cite{APWW98}. They proved that the sequences $\{|\mathbf{t}_{n}^{m}|(\bmod 2)\}_{n\geq0}$ are $2$-automatic. In the same way, Wen, Wu \cite{WW} and Guo, Wen \cite{GW} respectively studied the  the Hankel determinants of the Cantor sequence $\mathbf{c}$ and the differences of Thue-Morse $\Delta^k(\mathbf{t})$. They proved that the sequences $\{|\mathbf{c}_{n}^{m}|(\bmod 3)\}_{n\geq0}$ are $3$-automatic, and the the sequences $\{|\mathbf{\Delta^k(\mathbf{t})}_{n}^{m}|(\bmod 2)\}_{n\geq0}$ are $2$-automatic.

Here, we point out  that  if $\{u(n)\}_{n\geq0}$ is $k$-regular over $\mathbb{Z}$, then $\{u(n) (\bmod m)\}_{n\geq0}$ is $k$-automatic for any $m\geq1$. But the converse does not hold. For example, the sequence $\{2^n(\bmod m)\}_{n\geq0}$ is $k$-automatic for any $m\geq1,k\geq2$, but the sequence $\{2^n\}_{n\geq0}$ is not $k$-regular for any $k\geq2$, please see \cite{AS1992}. Hence, although there are many sequences are either $k$-automatic or $k$-regular in \cite{CRS}, it is often quite challenging to determine the automaticity of the Hankel determinant sequences $\{|\mathbf{u}_{n}^{m}|\}_{n\geq0}$.

%for a $k$-automatic integer sequence $\mathbf{u}=\{u(n)\}_{n\geq0}$ and a fixed $m\geq0$, it is often quite challenging to determine whether the sequence $\{|\mathbf{u}_{n}^{m}|\}_{n\geq0}$ is $k$-regular or not.

\smallskip

In this paper, we consider the Hankel determinants of the characteristic sequence of powers 2 $$\mathbf{p}=\{p(n)\}_{n\geq0}=0110100010000\cdots,$$
where $p(n)=1$ if $n=2^k$ for some $k\geq0$ and $p(n)=0$ otherwise.

Let $d(m,n)$ denote the $|\mathbf{p}_{n}^{m}|$ for $m\geq0,n\geq1$.  Using polynomial generated sequnces, we prove that
\begin{thm}\label{regular}
The sequence $\{d(0,n)\}_{n\geq0}$ is $2$-regular;
The sequences $\{d(2k,n)\}_{n\geq0}$ are $2$-automatic for all $k\geq1$; The sequences $\{d(2k+1,n)\}_{n\geq0}$ are periodic for all $k\geq0$.
\end{thm}
\smallskip

The sequence $\mathbf{p}$ is a $2$-automatic sequence, more about the sequence $\mathbf{p}$,  please see \cite{BHMV}.
%Let  $\sigma:a\mapsto ab,b\mapsto bc, c\mapsto cc$ be a morphism over $\{a,b,c\}$ and $\tau:a\mapsto 0, b\mapsto1, c\mapsto 0$ be a coding.  Then, $\mathbf{p}=\tau(\sigma^{\infty}(a))$. Moreover, it can be generated by the recurrence formulae:
%\begin{equation*}
%p(0)=0, p(1)=1, p(2n)=p(n), p(2n+1)=0,~(n\geq1).
%\end{equation*}
Let $C_n=\frac{1}{n+1}{2n \choose n}$ be a Catalan number, then $C_n(\bmod~2)=p(n+1)$ for all $n\geq0$. Recently, using permutation, Cigler \cite{Cigler} studied the Hankel determinants of the sequence $\{C_n(\bmod 2)\}_{n\geq0}=\{p(n+1)\}_{n\geq0}.$
By computer  experiments, they have a conjecture which have been proved by us. We state it  as follows.
\begin{thm}[Cigler \cite{Cigler}]\label{conjecture}

Assume $m$ is an integer with $2^k<m\leq2^{k+1}$ for some $k\geq1$.
\begin{enumerate}
  \item If $m=2r+1$, then $d(m,2^{k+1}n)=1$ and $d(m,2^{k+1}n-m+1)=(-1)^{r}.$
  \item If $m=2r$, then $d(m,2^{k+1}n)=d(2,2^{k+1}n)$ and $d(m,2^{k+1}n-m+1)=(-1)^{n+\epsilon_{r}}d(2,2^{k+1}n-m+1),$ where $\{\epsilon_{r}\}_{r\geq1}$ is  a sequence over $\{0,1\}$ defined by
\begin{equation*}\label{r}
\epsilon_{1}=1,~\epsilon_{2r}= (\epsilon_{r}+r)~\bmod2,~ \epsilon_{2r+1}=\epsilon_{r+1}, ~(r\geq1).
\end{equation*}
\end{enumerate}
\end{thm}
\smallskip

%In this paper, we will answer this conjecture by a precise recurrence formula of the determinants $d(m,n)$ for $m\geq0,n\geq1$ in the last section.
The paper is organized as follows.
In Section \ref{sec.2},  we recall some notation briefly.
In Section \ref{sec.3},  we  give a proof of Theorem \ref{polynomial1} and Theorem \ref{polynomial2}.
In the lase section, we study the Hankel determinants $d(m,n)$ and answer this conjecture by a precise recurrence formula of the determinants $d(m,n)$ for $m\geq0,n\geq1$. At last, we give a proof of Theorem \ref{regular}.

%%%%%%%%%%%%%%%%%%%%%%%%%%%%%%%%%%%%%%%%%%%%%%%%%%%%%%%%%%%%%%%%%%%%%%%%%%%%
\section{Preliminaries}\label{sec.2}
In this section, we briefly recall some notation and theorems. More notation, please see \cite{AS03}.

In this paper, we always denote the set of non-negative integers by $\mathbb{N}$ and denote the set of integers by $\mathbb{Z}$. For any two set $A, B$ and number $c$, define $A+B=\{a+b:a\in A,b\in B\}, AB=\{ab:a\in A,b\in B\}$ and $cA=\{ca:a\in A\}.$ We define $A^0=1$ and $A^i=AA^{i-1}, A^{-i}=\emptyset$ for any $i\geq1$. Assume $f$ is map defined on $A$, then the \emph{image} of $A$ under $f$ is denoted by $f(A)$, i.e., $f(A)=\{f(a): a\in A\}.$ In particular, if $f(A)\subset A$, then we say that the set $A$ is \emph{invariant}  under the map $f$.

We always call a finite set $\Sigma$ \emph{alphabet} and its elements \emph{letters}.  A word is made up of letters by the operation of concatenation. We denote the set of all finite words by $\Sigma^*$, including empty word $\epsilon$. Together with the operation of concatenation,  $\Sigma^*$ forms a free monoid.  An \emph{infinite sequence}, denote by $\mathbf{u}=\{u(n)\}_{n\geq0}=u(0)u(1)u(2)\cdots$, is a map from $\mathbb{N}$ to $\Sigma$. The set of all infinite sequences over $\Sigma$ is denoted by $\Sigma^{\mathbb{N}}.$

Among the infinite sequences, $k$-automatic and $k$-regular sequences satisfy a variety of useful properties. We recall some results from \cite{AS1992,jp}.
\begin{thm}\label{relation}(\cite{AS1992}, Theorem 2.3)
A sequence is $k$-regular and takes on only finitely many values if and only if it is $k$-automatic.
\end{thm}

\begin{thm}\label{cloure}(\cite{AS1992}, Theorem 2.5)
Let $\{u(n)\}_{n\geq0}$ and $\{v(n)\}_{n\geq0}$ be $k$-regular sequences. Then so are $\{u(n)+v(n)\}_{n\geq0}$, $\{u(n)v(n)\}_{n\geq0}$ and $\{cu(n)\}_{n\geq0}$ for any $c$.

\end{thm}

\begin{thm}\label{non-regular}(\cite{AS1992}, Theorem 2.10)
Let $\{u(n)\}_{n\geq0}$ be a k-regular sequence with values in $\mathbb{C}$, the set of complex numbers. Then there exists a constant $c$ such that $u(n)=\mathcal{O}(n^c).$
\end{thm}

\begin{thm}\label{linear}(\cite{jp}, Theorem 6)
Let $\{u(n)\}_{n\geq0}$ be a sequence with values in a Noetherian ring $R$. Suppose there exist integers $k\geq2, t, r, n_0$ such that each sequence $\{u(k^{t+1}n+e)\}_{n\geq n_0}$ for $0\leq e<k^{t+1}$ is a linear combination of the sequences $\{u(k^in+j)\}_{n\geq n_0}$ with $0\leq i\leq t, 0\leq j<k^i$ and the sequences $\{u(n+p)\}_{n\geq n_0}$ with $0\leq p\leq r$. Then the sequence $\{u(n)\}_{n\geq0}$ is $k$-regular.
\end{thm}
%%%%%%%%%%%%%%%%%%%%%%%%%%%%%%%%%%%%%%%%%%%%%%%%%%%%%%%%%%%%%%%%%%%%%%%%%%%%

\section{Polynomial generated sequence}\label{sec.3}
In this section, we study the polynomial generated sequences and give a proof of Theorem \ref{polynomial1} and Theorem \ref{polynomial2}.
\smallskip

Let $f_i(x)=\mathbf{a}_ix+\mathbf{b}_i$ be sequence polynomials, where $\mathbf{a}_i=\{a_i(n)\}_{n\geq0}, \mathbf{b}_i=\{b_i(n)\}_{n\geq0}$ for $0\leq i< k.$
Let
$$\mathcal{A}=\mathcal{K}_{k}(\mathbf{a}_0)\cup\mathcal{K}_{k}(\mathbf{a}_1)\cup\cdots\cup\mathcal{K}_{k}(\mathbf{a}_{k-1})$$ and $$\mathcal{B}=\mathcal{K}_{k}(\mathbf{b}_0)\cup\mathcal{K}_{k}(\mathbf{b}_1)\cup\cdots\cup\mathcal{K}_{k}(\mathbf{b}_{k-1}).$$

\smallskip
To prove Theorem \ref{polynomial1}, we need the following lemma.
\begin{lem}\label{finitely}
Let $\mathcal{S}$ be a finite set of integer sequences which take finitely many values, then the set
$\bigcup_{i\geq0}\mathcal{S}^i$
is finitely generated. Moreover, if the sequences of $\mathcal{S}$ take values in $\{-1,0,1\}$, then $\bigcup_{i\geq0}\mathcal{S}^i$ is a finite set.
\end{lem}
\begin{proof}
Let $\mathbf{u}$ be an integer sequence over the alphabet $\{a_0,a_1,\cdots,a_M\}$, we define the characteristic sequences $\chi_\mathbf{u}^j$ for $0\leq j\leq M$, by
  $$\chi_\mathbf{u}^j(n)=\begin{cases}
   1   & \text{if $u(n)=a_j$ }, \\
   0  & \text{otherwise}.
\end{cases}$$
 Then $\mathbf{u}$ is a $\mathbb{Z}$-linear combination of a finite set $\{\chi_\mathbf{u}^0,\chi_\mathbf{u}^1,\cdots,\chi_\mathbf{u}^M\}$, i.e.,
  $$\mathbf{u}=a_0\chi_\mathbf{u}^0+a_1\chi_\mathbf{u}^1+\cdots+a_M\chi_\mathbf{u}^M.$$

  Note that $\chi_\mathbf{u}^j$ are binary sequences, so we have $(\chi_\mathbf{u}^j)^k=\chi_\mathbf{u}^j$ for all $k\geq1$. Hence, $\mathbf{u}^k$ is a $\mathbb{Z}$-linear combination of at most $2^{M+1}$ sequences which are of the  form
 $$(\chi_\mathbf{u}^0)^{i_0}(\chi_\mathbf{u}^1)^{i_1}\cdots(\chi_\mathbf{u}^M)^{i_M},$$ where $i_j\in\{0,1\}$ for $0\leq j\leq M$.

Now, assume $\mathcal{S}=\{\mathbf{u}_0,\mathbf{u}_1,\cdots,\mathbf{u}_N\}$. Then, for $i\geq0$, we have
$$\mathcal{S}^i=\{\mathbf{u}_0^{n_0}\mathbf{u}_1^{n_1}\cdots\mathbf{u}_N^{n_N}: n_j\geq0 ~\text{for all}~0\leq j\leq N~\text{and}~n_0+n_1+\cdots+n_N=i\}.$$
So, each sequence $\mathbf{u}_i^{n_i}$ of $\mathcal{S}^i$ is a $\mathbb{Z}$-linear combination of a finite number of sequences. Hence, the set $\bigcup_{i\geq0}\mathcal{S}^i$ is finitely generated.

In particular, if a sequence $\mathbf{u}$ takes values in $\{-1,0,1\}$, then $\mathbf{u}^k\in\{\mathbf{u},\mathbf{u}^2\}$ for all $k\geq1$. Hence, if the sequences of $\mathcal{S}$ take values in $\{-1,0,1\}$, then $\mathbf{u}_i^k\in\{\mathbf{u}_i,\mathbf{u}_i^2\}$ for all $k\geq1$ for each $0\leq i\leq N$. Thus, $\bigcup_{i\geq0}\mathcal{S}^i$ is a finite set.
\end{proof}

\begin{proof}[Proof of Theorem \ref{polynomial1}]
Assume $\mathbf{u}=\mathcal{G}(f_0,f_1,\cdots,f_{k-1})$. Then,  we have $u(kn+i)=a_i(n)u(n)+b_i(n)$ for $0\leq i< k, n\geq0$. Hence,
$$\mathcal{K}_k(\mathbf{u})\subset\bigcup_{i\geq0}\left(\mathcal{A}^i\mathbf{u}+\mathcal{B}\sum_{s=0}^{i-1}\mathcal{A}^{s}\right).$$

Since the sequences $\mathbf{a}_i=\{a_i(n)\}_{n\geq0}$ are $k$-automatic and $\mathbf{b}_i=\{b_i(n)\}_{n\geq0}$ are $k$-regular for $0\leq i< k$,  we know that the set $\mathcal{A}$ is finite and the set $\mathcal{B}$ is finitely generated. Hence, by Lemma \ref{finitely},  $\bigcup_{i\geq0}\mathcal{A}^i\mathbf{u}$ and $\bigcup_{i\geq0}\mathcal{B}(\sum_{s=0}^{i-1}\mathcal{A}^{s})$ are finitely generated by some finite set. Thus, the $k$-kernel $\mathcal{K}_k(\mathbf{u})$ is finitely generated, which implies that the polynomial generated sequence $\mathcal{G}(f_0,f_1,\cdots,f_{k-1})$ is $k$-regular.
\end{proof}

\smallskip

Now, using  Lemma \ref{finitely}, we  prove Theorem \ref{polynomial2}.
\begin{proof}[Proof of Theorem \ref{polynomial2}]
Let  $\mathbf{u}=\{u(n)\}_{n\geq0}$ be a polynomial generated sequence with polynomials $f_0,f_1,\cdots,f_{k-1}$ and shifts $S_i$ for $0\leq i<k.$
Since $S_i\in\{S^j:j\geq0\}$, we assume $S_i=S^{j_i}$, where $j_i\geq0$ are integers for $0\leq i<k$.
By definition \ref{defi2}, for any $0\leq i<k$ and $n\geq N$, we have
\begin{equation}
\label{shift-equation}
u(kn+i)=a_i(n)u(n+j_i)+b_i(n).
\end{equation}

Note that there are at most $kN$ terms of $\mathbf{u}$ which do not satisfy above formula (\ref{shift-equation}). For $0\leq i<k$, define a ultimately periodic sequence $\mathbf{c}_i=\{c_i(n)\}_{n\geq0}$ by
$$c_i(n)=\begin{cases}
u(kn+i)- a_i(n)u(n+j_i)-b_i(n),  & \text{if $0\leq n<N$}, \\
 0,     & \text{otherwise}.
\end{cases}$$
Then, for all $n\geq0$,
$$u(kn+i)=a_i(n)u(n+j_i)+b_i(n)+c_i(n),~(0\leq i<k).$$

Let $r=max\{j_i:0\leq i<k\}$ and $Q$ is a fixed number satisfying $Q \geq \frac{k(1+r)}{k-1}$. Define  sets
\begin{eqnarray*}
\mathcal{U}     & = & \Big\{\{u(n+t)\}_{n\geq0}:0\leq t\leq Q\Big\}, \\
\mathcal{A}_Q & = & \bigcup_{s=0}^{k-1}\left\{\{a_s(k^i(n+t)+j)\}_{n\geq0}:i\geq0, 0\leq j<k^i, 0\leq t\leq Q\right\},\\
\mathcal{B}_Q & = & \bigcup_{s=0}^{k-1}\left\{\{b_s(k^i(n+t)+j)\}_{n\geq0}:i\geq0, 0\leq j<k^i, 0\leq t\leq Q\right\},\\
\mathcal{C}_Q & = & \bigcup_{s=0}^{k-1}\left\{\{c_s(k^i(n+t)+j)\}_{n\geq0}:i\geq0, 0\leq j<k^i, 0\leq t\leq Q\right\}.
\end{eqnarray*}

We claim that,
$$\mathcal{K}_k(\mathbf{u})\subset\bigcup_{i\geq0}\left(\mathcal{A}_Q^i\mathcal{U}+(\mathcal{B_Q}+\mathcal{C_Q})\sum_{s=0}^{i-1}\mathcal{A}_Q^{s}\right).$$

To prove this claim, we need some maps. For $0\leq \ell<k$, define $\psi_{\ell}:\{w(n)\}_{n\geq0}\rightarrow \{w(kn+\ell)\}_{n\geq0}$. Then, we only need to show that $\psi_{\ell}(\mathcal{U})\subset\mathcal{A}_Q\mathcal{U}+\mathcal{B}_Q+\mathcal{C}_Q$ and  the sets $\mathcal{A}_Q,\mathcal{B}_Q$ and $\mathcal{C}_Q$ are invariant under the maps $\psi_{\ell}.$

For any $0\leq t\leq Q$ and $0\leq \ell<k$, assume $t+\ell=kx+y$, where $0\leq y<k$.  So, by the choice of $Q$ and note that $0\leq j_y\leq r$, we have
$$0\leq x\leq  x+j_y\leq \frac{t+\ell}{k}+r\leq\frac{k-1+Q}{k}+r\leq Q.$$
Then,  we have the following cases.
\begin{itemize}
  \item  $u(kn+\ell+t)=u(k(n+x)+y)=a_y(n+x)u(n+x+j_y)+b_y(n+x)+c_y(n+x)$, so we have $\psi_{\ell}(\mathcal{U})\subset\mathcal{A}_Q\mathcal{U}+\mathcal{B}_Q+\mathcal{C}_Q;$
  \item For any $i\geq0, 0\leq j<k^i$, we have $a_s(k^i(kn+\ell+t)+j)=a_s(k^i(k(n+x)+y)+j)=a_s(k^{i+1}(n+x)+k^iy+j)$. Note that $0\leq k^iy+j<k^{i+1},$ we have $\psi_{\ell}(\mathcal{A}_Q)\subset\mathcal{A}_Q.$ In the same way, $\psi_{\ell}(\mathcal{B}_Q)\subset\mathcal{B}_Q$ and $\psi_{\ell}(\mathcal{B}_Q)\subset\mathcal{B}_Q.$
\end{itemize}
Hence, our claim holds.

If  the sequences $\mathbf{a}_i=\{a_i(n)\}_{n\geq0}$ are $k$-automatic and $\mathbf{b}_i=\{b_i(n)\}_{n\geq0}$ are $k$-regular for $0\leq i< k$, then the set $\mathcal{A}_Q$ is finite and the set $\mathcal{B}_Q$ is finitely generated. Note that the sets $\mathcal{U}$ and $\mathcal{C}_Q$ are finite, by Lemma \ref{finitely}, there exists a finite set $\mathcal{S}$ such that the sets $$\bigcup_{i\geq0}\left(\mathcal{A}_Q^i\mathcal{U}+(\mathcal{B_Q}+\mathcal{C_Q})\sum_{s=0}^{i-1}\mathcal{A}_Q^{s}\right)$$ are $\mathbb{Z}$-linear combination of $\mathcal{S}$. Hence, the $k$-kernel $\mathcal{K}_k(\mathbf{u})$ is finitely generated and our theorem follows.
\end{proof}

\begin{remark}
Note from the proof of Theorem \ref{polynomial2} that the condition of Definition \ref{poly1}  can be weakened by
$u(kn+i)=f_i(\mathbf{u})(n)~(n\geq N)$
for some integer $N\geq0$. Denote
$$\mathcal{G}_{N}(f_0,f_1,\cdots,f_{k-1})=\Big\{\{u(n)\}_{n\geq0}: u(kn+i)=f_i(\mathbf{u})(n), n\geq N ~\text{for}~0\leq i\leq k\Big\}.$$
If $f_i=\mathbf{a}_ix+\mathbf{b}_i$, where $\mathbf{a}_i$ are $k$-automatic and $\mathbf{b}_i$ are $k$-regular for $0\leq i\leq k$, then for every $N\geq0$, the sequences of $\mathcal{G}_{N}(f_0,f_1,\cdots,f_{k-1})$ are $k$-regular.
\end{remark}

We end this section by some examples. Example \ref{degree} implies that the condition that all polynomials are degree $1$ in Theorem \ref{polynomial1}, is necessary. Example \ref{a-regular} shows that if  we replace the $k$-automatic condition of the sequence $\mathbf{a}_i$ by a $k$-regular condition in Theorem \ref{polynomial1},  then the polynomial generated  sequence maybe not $k$-regular.

\begin{exam}\label{degree}
The sequences in $\mathcal{G}(x^2,x+1)$ are not $2$-regular. Assume $\{u(n)\}_{n\geq0}\in\mathcal{G}(x^2,x+1)$, and
$$u(2n)=u(n)^2, ~u(2n+1)=u(n)+1.$$
If $u(0)=0$, then $u(3)=2$. If $u(0)=1$, then $u(2)=2$.  In either case, $a\in\{2,3\}$, we have
$$\frac{\log_2(u(a\cdot 2^k))}{\log_2(a\cdot 2^k)}=\frac{2^k\log_2(u(a))}{k(\log_2a+1)}=\frac{2^k}{(1+\log_2a)k}\rightarrow \infty, ~(k\rightarrow \infty).$$
Hence, by Theorem \ref{non-regular}, the sequences in $\mathcal{G}(x^2,x+1)$ are not $2$-regular.
\end{exam}

\begin{exam}\label{a-regular}
The sequence $\mathcal{G}(\mathbf{v} x,x+1)$ is not $2$-regular, where $\mathbf{v}=\{n\}_{n\geq0}$ is a $2$-regular sequence. Assume $\{u(n)\}_{n\geq0}=\mathcal{G}(\mathbf{v} x,x+1)$, then
$$u(2n)=nu(n), ~u(2n+1)=u(n)+1.$$
It is easy to check that $u(2^kn)=2^{k-1+k-2+\cdots+1}n^ku(n)=2^{k(k-1)/2}n^ku(n).$
Since $u(3)=2$ and
$$\frac{\log_2(u(3\cdot 2^k))}{\log_2(3\cdot 2^k)}=\frac{k(k-1)/2+k\log_23+1}{k(\log_23+1)}\rightarrow \infty, ~(k\rightarrow \infty)$$
which implies the sequence $\mathcal{G}(\mathbf{v} x,x+1)$ is not $2$-regular, by Theorem \ref{non-regular}.

\end{exam}

\begin{exam}
Let $\mathbf{t}=\{t(n)\}_{n\geq0}=01101001\cdots$ be the Thue-Morse sequence.  Clearly that $t(2n)=t(n), t(2n+1)=1-t(n)$ for all $n\geq0$.  Let $\mathbf{u}=\{u(n)\}_{n\geq0}=\mathcal{G}(\mathbf{t}x+1,x+\mathbf{t})$.  Then, for $n\geq0$,
\begin{eqnarray*}
&&u(0)=u(1)=1,\\
&&u(4n)=u(4n+1)  =  u(2n)+u(2n+1)-u(n), \\
&&u(4n+2)= -u(2n)+u(n)+2,\\
&&u(4n+3)=  u(n)+1.
\end{eqnarray*}
Hence, by Theorem \ref{linear}, $\mathcal{G}(\mathbf{t}x+1,x+\mathbf{t})$ is a 2-regular sequence.
\end{exam}

%%%%%%%%%%%%%%%%%%%%%%%%%%%%%%%%%%%%%%%%%%%%%%%%%%%%%%%%%%%%%%%%%%%%%%%%%%%%
\section{Application}\label{sec.4}
In this section, we study the Hankel determinants $d(m,n)$ and obtain the recurrence formulae of the determinants $d(m,n)$ for $m\geq0,n\geq1$.
By these formulae,  we prove the conjecture of Cigler.  Moreover, we find that the Hankel determinant sequence $\{d(0,n)\}_{n\geq0}$ is a  polynomial generated sequence with shift.  Using Corollary \ref{coro-poly2},  we give a proof of Theorem \ref{regular} at last.

%@@@@@@@@@@@@@@@@@@@@@@@@@@@@@@@@@@@@@@@@@@@@@@@@@@@@@@@@@@@@@@@@@@@@@

%\subsection{Deterninantes}
Let $\mathbf{p}=\{p(n)\}_{n\geq0}$ be the characteristic sequence of powers 2. Recall that  the sequence $\{p(n)\}_{n\geq0}$ can be generated by the recurrence formula:
\begin{equation}\label{equation1}
p(0)=0, p(1)=1, p(2n)=p(n), p(2n+1)=0,~(n\geq1).
\end{equation}

Let $d(m,n)=|\mathbf{p}_n^m|=\det(p(i+j+m))_{i,j=0}^{n-1}$ for any $m\geq0,n\geq1$,  then we have the following lemma which plays important role in this paper.
\begin{lem}\label{formula1}
For any $m\geq1,n\geq1$, we have
\begin{enumerate}
  \item $d(0,2n)=d(0,n)d(1,n)-d(2,n-1)d(3,n-1)$,
  \item $d(0,2n+1)=d(0,n+1)d(1,n)-d(2,n)d(3,n-1)$,
  \item $d(1,2n)=(-1)^{n}d^{2}(1,n)$,
  \item $d(1,2n+1)=(-1)^nd^2(2,n),$
  \item $d(2m,2n)=d(m,n)d(m+1,n)$,
  \item $d(2m,2n+1)=d(m,n+1)d(m+1,n)$,
  \item $d(2m+1,2n)=(-1)^nd^2(m+1,n)$,
  \item $d(2m+1,2n+1)=0$.
\end{enumerate}
Here,  we define $d(2,0)=d(3,0)=1.$
\end{lem}
\begin{proof}
For each $n$-order square matrix $M=(m_{i,j})_{1\leq i,j\leq n}$, there exists a matrix $U$ with $|U|=\pm1$ such that
\begin{equation}\label{equation2}
    UMU^{t}=\left(
              \begin{array}{cc}
                (m_{2i-1,2j-1})_{\begin {subarray}{c} 1\leq i\leq\mu \\1\leq j\leq\mu\end {subarray}} & (m_{2i-1,2j})_{\begin {subarray}{c} 1\leq i\leq\mu \\1\leq j\leq\nu\end {subarray}} \\
                (m_{2i,2j-1})_{\begin {subarray}{c} 1\leq i\leq\nu \\1\leq j\leq\mu\end {subarray}} & (m_{2i,2j})_{\begin {subarray}{c} 1\leq i\leq\nu \\1\leq j\leq\nu\end {subarray}} \\
              \end{array}
            \right),
\end{equation}
where $\mu=\lfloor\frac{1}{2}(n+1)\rfloor$ and $\nu=\lfloor\frac{1}{2}n\rfloor$ and $U^{t}$ denote the transposed matrix of $U$.

$(1)$ By Formula (\ref{equation1}) and (\ref{equation2}), we have
\begin{equation*}
    U\mathbf{p}_{2n}^{0}U^{t}=\left(
                                    \begin{array}{cc}
                                      \mathbf{p}_{n}^{0} & A_{n,n} \\
                                      A_{n,n} & \mathbf{p}_{n}^{1} \\
                                    \end{array}\mathbb{}
                                  \right)
\end{equation*}
where $A_{m,n}=(a_{ij})_{1\leq i\leq m,1\leq j\leq n}$ denote  the $m\times n$ matrix with all entries are zero except $a_{11}=1$.
Hence,
\begin{equation*}
   |\mathbf{p}_{2n}^{0}| =   \left| \begin{array}{cc}
                                      \mathbf{p}_{n}^{0} & A_{n,n} \\
                                      A_{n,n} & \mathbf{p}_{n}^{1} \\
                                    \end{array}
                                  \right| \\
   =  |\mathbf{p}_{n}^{0}||\mathbf{p}_{n}^{1}|-|\mathbf{p}_{n-1}^{2}||\mathbf{p}_{n-1}^{3}|.
   \end{equation*}

$(2)$ By Formula (\ref{equation1}) and (\ref{equation2}), we have
\begin{equation*}
   |\mathbf{p}_{2n+1}^{0}| =   \left|  \begin{array}{cc}
                                      \mathbf{p}_{n+1}^{0} & A_{n+1,n} \\
                                      A_{n,n+1} & \mathbf{p}_{n}^{1} \\
                                    \end{array}
                                  \right| \\
   =  |\mathbf{p}_{n+1}^{0}||\mathbf{p}_{n}^{1}|-|\mathbf{p}_{n}^{2}||\mathbf{p}_{n-1}^{3}|.
\end{equation*}

$(3)$ By Formula (\ref{equation1}) and (\ref{equation2}), we have
\begin{equation*}
   |\mathbf{p}_{2n}^{1}| =   \left|  \begin{array}{cc}
                                       A_{n,n} & \mathbf{p}_{n}^{1}  \\
                                       \mathbf{p}_{n}^{1}& \mathbf{0}_{n,n} \\
                                    \end{array}
                                  \right| \\
   =  (-1)^{n}|\mathbf{p}_{n}^{1}|^{2}.
\end{equation*}
where $\mathbf{0}_{m,n}$ denote the $m\times n$ matrix with all entries are zero.

$(4)$ By Formula (\ref{equation1}) and (\ref{equation2}), we have
\begin{equation*}
   |\mathbf{p}_{2n+1}^{1}| =   \left|                                      \begin{array}{cc}
                                       A_{n+1,n+1} & \mathbf{p}_{n+1,n}^{1}  \\
                                       \mathbf{p}_{n,n+1}^{1}& \mathbf{0}_{n,n} \\
                                    \end{array}
                                  \right| \\
   =  (-1)^{n}|\mathbf{p}_{n}^{2}|^{2}.
\end{equation*}

$(5)$ By Formula (\ref{equation1}) and (\ref{equation2}), we have
\begin{equation*}
   |\mathbf{p}_{2n}^{2m}| =   \left|  \begin{array}{cc}
                                        \mathbf{p}_{n}^{m} & \mathbf{0}_{n,n}   \\
                                        \mathbf{0}_{n,n} & \mathbf{p}_{n}^{m+1} \\
                                    \end{array}
                                  \right| \\
   =|\mathbf{p}_{n}^{m}||\mathbf{p}_{n}^{m+1}|.
\end{equation*}

$(6)$ By Formula (\ref{equation1}) and (\ref{equation2}), we have
\begin{equation*}
   |\mathbf{p}_{2n}^{2m+1}| =   \left|  \begin{array}{cc}
                                        \mathbf{0}_{n,n} & \mathbf{p}_{n}^{m+1}    \\
                                        \mathbf{p}_{n}^{m+1} & \mathbf{0}_{n,n}\\
                                    \end{array}
                                  \right| \\
   =(-1)^{n}|\mathbf{p}_{n}^{m+1}|^{2}.
\end{equation*}

$(7)$ By Formula (\ref{equation1}) and (\ref{equation2}), we have
\begin{equation*}
   |\mathbf{p}_{2n+1}^{2m}| =   \left|  \begin{array}{cc}
                                        \mathbf{p}_{n+1}^{m} & \mathbf{0}_{n+1,n}   \\
                                        \mathbf{0}_{n,n+1} & \mathbf{p}_{n}^{m+1} \\
                                    \end{array}
                                  \right| \\
   =|\mathbf{p}_{n+1}^{m}||\mathbf{p}_{n}^{m+1}|.
\end{equation*}

$(8)$ By Formula (\ref{equation1}) and (\ref{equation2}), we have
\begin{equation*}
   |\mathbf{p}_{2n+1}^{2m+1}| =   \left|  \begin{array}{cc}
                                        \mathbf{0}_{n+1,n+1} & \mathbf{p}_{n+1,n}^{m+1}    \\
                                        \mathbf{p}_{n,n+1}^{m+1} & \mathbf{0}_{n,n}\\
                                    \end{array}
                                  \right| \\
   =0.
\end{equation*}
\end{proof}

\begin{remark}
Define $d(m,0)=1,d(m,-1)=0$ for all $m\geq0$,  then Formulae of Lemma \ref{formula1} hold for $n\geq0$.
\end{remark}

\begin{prop}\label{formula2}
For any $n\geq1$, $d(1,n),d(2,n)\in\{-1,1\},d(m,n)\in\{-1,0,1\}~(m\geq3)$. Moreover, $d(1,n)=(-1)^{\lfloor \frac{n}{2}\rfloor}$.
\end{prop}
\begin{proof}
We first prove that $d(1,n),d(2,n)\in\{-1,1\}$ by induction on $n$.  It is easy to check that $d(1,1)=-d(1,2)=-d(1,3)=1,d(2,1)=d(2,2)=-d(2,3)=1.$ Now, assume that $d(1,n),d(2,n)\in\{-1,1\}$ for all $n<2^{k}$ with $k\geq1$. Then, for any $2^{k}\leq n<2^{k+1}$, there exists an integer $m<2^{k}$  such that $n=2m$ or $n=2m+1$. Moreover,
\begin{itemize}
  \item $d(1,n)=d(1,2m)=(-1)^md^2(1,m)\in\{-1,1\},$
  \item $d(1,n)=d(1,2m+1)=(-1)^md^2(2,m)\in\{-1,1\},$
  \item $d(2,n)=d(1,2m)=d(1,m)d(2,m)\in\{-1,1\},$
  \item $d(2,n)=d(1,2m+1)=d(1,m+1)d(2,m)\in\{-1,1\}.$
\end{itemize}
Hence, $d(1,n),d(2,n)\in\{-1,1\}$ for all $n\geq1$.

\medskip

Now, assume there exists an integer $k$ such that $d(m,n)\in\{-1,0,1\}$ for all $m\leq2^k,n\geq1$. We need to prove the conclusion hold for $m\leq 2^{k+1}$. Since $m=2s$ or $m=2s+1$ for some $s\leq2^k$, by Lemma \ref{formula1} and the hypothesis, we have
\begin{itemize}
  \item $d(m,2n)=d(2s,2n)=d(s,n)d(s+1,n)\in\{-1,0,1\}$,
  \item $d(m,2n+1)=d(2s,2n+1)=d(s,n+1)d(s+1,n)\in\{-1,0,1\}$,
  \item $d(m,2n)=d(2s+1,2n)=(-1)^nd^2(s+1,n)\in\{-1,0,1\}$,
  \item $d(m,2n+1)=d(2s+1,2n+1)=0\in\{-1,0,1\}$.
\end{itemize}
Thus, $d(m,n)\in\{-1,0,1\}$ for all $m,n\geq1.$

By Lemma \ref{formula1}, it follows that
$d(1,n)=(-1)^{\lfloor \frac{n}{2}\rfloor },$
which completes this proof.
\end{proof}

\medskip

The following two propositions have been proved by Cigler in \cite{Cigler}.  Here, we give another proofs of them. Our method which is different from Cigler mainly depends on the recurrence formulae.
The first proposition gives a description of the sequence $\{d(2,n)\}_{n\geq0}$. The second proposition gives a description of the sequence $\{d(m,n)\}_{n\geq0}$ for all $m\geq3$.

\begin{prop}\label{formula3}
If $2^k\leq n<2^{k+1}$ for some integer $k\geq2$, then
$$d(2,n)=\begin{cases}
      -d(2,n-2^k)& \text{ if $2^k\leq n<2^k+2^{k-1}$}, \\
      d(2,n-2^k)& \text{if $2^k+2^{k-1}\leq n<2^{k+1}$}.
\end{cases}$$
\end{prop}
\begin{proof}
For $k=2$, the assertions above can be checked directly.  Assume the proposition is true for $k\leq N$. Now, we discuss the case $k=N+1$.

If $2^{N+1}\leq n<2^{N+1}+2^{N}$ and $n=2^{N+1}+2m$, then $0\leq m<2^{N-1}.$ By Lemma \ref{formula1} and the hypothesis, we have $d(2,2^{N}+m)=-d(2,m)$ and
       \begin{eqnarray*}
d(2,n) & = & d(2,2^{N+1}+2m) =d(1,2^{N}+m)d(2,2^{N}+m)\\
 & = & -d(1,2^{N}+m)d(2,m)=-d(1,m)d(2,m)\\
 & = & -d(2,2m)=-d(2,n-2^{N+1}).
\end{eqnarray*}
\iffalse
If $2^{N+1}\leq n<2^{N+1}+2^{N}$ and $n=2^{N+1}+2m+1$, then $0\leq m<2^{N-1}.$ By Lemma \ref{formula1} and the hypothesis, we have $d(2,2^{N}+m)=-d(2,m)$ and
       \begin{eqnarray*}
d(2,n) & = & d(2,2^{N+1}+2m+1) =d(1,2^{N}+m+1)d(2,2^{N}+m)\\
 & = & -d(1,2^{N}+m+1)d(2,m)=-d(1,m+1)d(2,m)\\
 & = & -d(2,2m+1)=-d(2,n-2^{N+1}).
\end{eqnarray*}
  \item If $2^{N+1}+2^{N}\leq n<2^{N+2}+2^{N}$ and $n=2^{N+1}+2m$, then $2^{N-1}\leq m<2^{N}.$ By Lemma \ref{formula1} and the hypothesis, we have $d(2,2^{N}+m)=d(2,m)$ and
       \begin{eqnarray*}
d(2,n) & = & d(2,2^{N+1}+2m) =d(1,2^{N}+m)d(2,2^{N}+m)\\
 & = & d(1,2^{N}+m)d(2,m)=d(1,m)d(2,m)\\
 & = & d(2,2m)=d(2,n-2^{N+1}).
\end{eqnarray*}
\item If $2^{N+1}+2^{N}\leq n<2^{N+2}+2^{N}$ and $n=2^{N+1}+2m+1$, then $2^{N-1}\leq m<2^{N}.$ By Lemma \ref{formula1} and the hypothesis, we have $d(2,2^{N}+m)=d(2,m)$ and
       \begin{eqnarray*}
d(2,n) & = & d(2,2^{N+1}+2m) =d(1,2^{N}+m)d(2,2^{N}+m)\\
 & = & d(1,2^{N}+m+1)d(2,m)=d(1,m+1)d(2,m)\\
 & = & d(2,2m+1)=d(2,n-2^{N+1}).
\end{eqnarray*}
\fi
The other ones can be obtained by the same method.
\end{proof}
\begin{remark}
In fact, Proposition \ref{formula3} gives a generation method of  the sequence $\{d(2,n)\}_{n\geq0}$. Let $A_0=11,B_0=1-1,A_n=A_{n-1}B_{n-1},B_n=\overline{A_{n-1}}B_{n-1}~(n\geq1), $ where the overbar is shorthand for the morphism that maps 1 to -1 and -1 to 1. Then, $$\{d(2,n)\}_{n\geq0}=\lim_{n\rightarrow\infty}A_n=A_0B_0\overline{A_0}B_0\overline{A_0}~\overline{B_0}~\overline{A_0}B_0\cdots=111-1-1-11-1-1-1-11\cdots.$$
\end{remark}

\medskip

%The following proposition gives a description of the sequence $\{d(m,n)\}_{n\geq0}$ for all $m\geq3$.
\begin{prop}\label{formula4}
If $2^k<m\leq2^{k+1}$ for some integer $k\geq1$, then
$$d(m,n)=\begin{cases}
    \pm1,  & \text{if $n\equiv0$ or $1-m ~(\bmod 2^{k+1})$ }, \\
    0,  & \text{otherwise}.
\end{cases}$$
\end{prop}
\begin{proof}
By (7), (8) of Lemma \ref{formula1} and Proposition \ref{formula2}, we have
$$d(3,2n)=(-1)^{n}d^{2}(2,n)=(-1)^{n}, d(3,2n+1)=0.$$
Then, by (5), (6) of Lemma \ref{formula1}, we have
\begin{itemize}
  \item $d(4,4n)=d(2,2n)d(3,2n)=(-1)^{n}d(2,2n)$,
  \item $d(4,4n+1)=d(2,2n+1)d(3,2n)=(-1)^{n}d(2,2n+1)$,
  \item $d(4,4n+2)=d(2,2n+1)d(3,2n+1)=0,$
  \item $d(4,4n+3)=d(2,2n+2)d(3,2n+1)=0.$
\end{itemize}
Hence,
$ d(3,n)\neq0  \Leftrightarrow   n\equiv 0~\text{or}~ 2 ~(\bmod 4),
d(4,n)\neq0  \Leftrightarrow   n\equiv 0~\text{or} ~1 ~(\bmod 4),
$
which implies that the conclusions hold for $k=1$.

\medskip

Now, assume the assertions hold for $k\leq N$, we need to prove the case $k=N+1$.  There are following cases to discuss.
\begin{itemize}
      \item If $2^{N+1}<m<2^{N+2}$ and $m=2r$ for some integer $r$, then $2^N<r,r+1\leq2^{N+1}$. By Lemma \ref{formula1} and the hypothesis, we have
             \begin{itemize}
                  \item $d(2r,2s)\neq0 \Leftrightarrow d(r,s)d(r+1,s)\neq0 \Leftrightarrow s\equiv 0 ~(\bmod 2^{N+1})\Leftrightarrow 2s\equiv 0 ~(\bmod 2^{N+2}).$
                  \item $d(2r,2s+1)\neq0 \Leftrightarrow d(r,s+1)d(r+1,s)\neq0 \Leftrightarrow s\equiv -r~ (\bmod 2^{N+1})\Leftrightarrow 2s+1\equiv 1-2r ~(\bmod 2^{N+2}).$
             \end{itemize}

     \item If $m=2^{N+2}$ and $m=2r$, then $r=2^{N+1},2^{N+1}<r+1\leq 2^{N+2}$. We have
$d(r+1,s)\neq0 \Leftrightarrow s\equiv 0 ~\text{or}~2^{N+1}~(\bmod 2^{N+2})\Leftrightarrow s\equiv 0 ~(\bmod  2^{N+1}).$
Hence, by Lemma \ref{formula1} and the hypothesis, we have
            \begin{itemize}
                \item $d(2r,2s)\neq0 \Leftrightarrow d(r,s)d(r+1,s)\neq0 \Leftrightarrow s\equiv 0 ~(\bmod  2^{N+1})\Leftrightarrow 2s\equiv 0~ (\bmod 2^{N+2}).$
                \item $d(2r,2s+1)\neq0 \Leftrightarrow d(r,s+1)d(r+1,s)\neq0 \Leftrightarrow s\equiv 0 ~(\bmod 2^{N+1})\Leftrightarrow 2s+1\equiv 1\equiv 1-2r (\bmod  2^{N+2}).$
           \end{itemize}

      \item If $2^{N+1}<m\leq2^{N+2}$ and $m=2r+1$ for some integer $r$, then $2^N<r+1\leq2^{N+1}$. By Lemma \ref{formula1} and the hypothesis, we have
   $$d(2r+1,2s)\neq0 \Leftrightarrow d(r+1,s)\neq0 \Leftrightarrow s\equiv 0~\text{or} ~-r~(\bmod  2^{N+1})\Leftrightarrow 2s\equiv 0~\text{or}~-2r ~(\bmod 2^{N+2}).$$

\end{itemize}

Thus,  if $2^{N+1}<m\leq2^{N+2}$, then
$$d(m,n)\neq0  \Leftrightarrow n\equiv 0~\text{or}~1-m~(\bmod 2^{N+2}),$$
which completes the proof.
\end{proof}

\medskip
Now, we give a proof of Theorem \ref{conjecture} which is an answer of Cigler's conjecture.
\begin{proof}[Proof of Theorem \ref{conjecture} ]
It is easy to check that that the two assertions hold for $k=1$. Assume the two assertions hold for $k\leq N~(N\geq1)$, it suffices to show that the assertions also hold for $k=N+1$. There are three cases to discuss.

\begin{itemize}
  \item If $2^{N+1}<m\leq2^{N+2}$ and $m=2r+1$, then, by (7) of Lemma \ref{formula1}, we have
     \begin{itemize}
       \item $d(m,2^{N+2}n)=d(2r+1,2^{N+2}n)=(-1)^{2^{N+1}n}=1,$
       \item $d(m,2^{N+2}n-m+1)=d(2r+1,2^{N+2}n-2r)=(-1)^{2^{N+1}n-r}=(-1)^{r}.$
     \end{itemize}

\medskip

  \item If $2^{N+1}<m\leq2^{N+2}$ and $m=4r+2$, then $2^{N}<2r+1, 2r+2\leq 2^{N+1}$.
       \begin{itemize}
       \item    Note that $d(2,2^{N+2}n)=d(2,2^{N+1}n)$. By (5) of Lemma \ref{formula1} and the hypothesis, we have     \vspace{-0.4em}
                \begin{eqnarray*}
               d(m,2^{N+2}n) & = & d(4r+2,2^{N+2}n)=d(2r+1,2^{N+1}n)d(2r+2,2^{N+2}n) \\
                & = & d(2,2^{N+1}n)=d(2,2^{N+2}n).\\   \end{eqnarray*}
                \vspace{-3em}
       \item  Note that $d(2,2^{N+2}n-4r-1)=(-1)^rd(2,2^{N+1}n-2r-1)$ and $\epsilon_{2r+1}=\epsilon_{r+1}.$  By (6) of Lemma \ref{formula1} and the hypothesis, we have  \vspace{-0.4em}
               \begin{eqnarray*}
               d(m,2^{N+2}n-m+1) & = & d(4r+2,2^{N+2}n-4r-1)\\
                & = & d(2r+1,2^{N+1}n-2r)d(2r+2,2^{N+1}n-2r-1) \\
                & = & (-1)^{r}(-1)^{n+\epsilon_{r+1}}d(2,2^{N+1}n-2r-1)\\
                & = & (-1)^{n+\epsilon_{2r+1}}d(2,2^{N+2}n-4r-1).
              \end{eqnarray*}
     \end{itemize}

\smallskip

  \item If $2^{N+1}<m\leq2^{N+2}$ and $m=4r$, then $2^{N}<2r\leq2^{N+1}, 2^{N}+1<2r+1\leq2^{N+1}+1$.
      \begin{itemize}
       \item  Note that $d(2,2^{N+2}n)=d(2,2^{N+1}n)$. By (5) of Lemma \ref{formula1}, the first case and the hypothesis, we have     \vspace{-0.5em}
              \begin{eqnarray*}
                  d(m,2^{N+2}n) & = & d(4r,2^{N+2}n)=d(2r,2^{N+1}n)d(2r+1,2^{N+2}n) \\
                                         & = &  d(2r,2^{N+1}n)=d(2,2^{N+1}n)=d(2,2^{N+2}n).
             \end{eqnarray*}
       \item  Note that $d(2,2^{N+2}n-4r+1)=(-1)^rd(2,2^{N+1}n-2r)=d(2,2^{N+1}n-2r+1)$ and $\epsilon_{2r}=(\epsilon_{r}+r)~\bmod 2.$ By (6) of Lemma \ref{formula1} and the hypothesis, we have  \vspace{-0.5em}
              \begin{eqnarray*}
                   d(m,2^{N+2}n-m+1) & = & d(4r,2^{N+2}n-4r+1) \\
                                                   & = & d(2r,2^{N+1}n-2r+1)d(2r+1,2^{N+2}n-2r)\\
                                                   & = & (-1)^{r} (-1)^{n+\epsilon_{r}}d(2,2^{N+1}n-2r+1)\\
                                                   & = & (-1)^{n+\epsilon_{2r}}d(2,2^{N+2}n-4r+1).
              \end{eqnarray*}
     \end{itemize}
\end{itemize}
Hence, our theorem follows.
\end{proof}

%\subsection{Automaticity and regularity}
%In the rest of this paper, we discuss further properties of the sequences $\{d(m,n)\}$ for all $m,n\geq0$.  We prove that the sequence $\{d(m,n)\}_{m\geq1,n\geq0}$ is two-dimensional $2$-automatic. Moreover, we prove that the sequence $\{d(0,n)\}_{n\geq0}$ is $2$-regular.

%@@@@@@@@@@@@@@@@@@@@@@@@@@@@@@@@@@@@@@@@@@@@@@@@@@@@@@@@@@@@@@@@@@@@@

By Lemma \ref{formula1}, Proposition \ref{formula2}-\ref{formula4} and Theorem \ref{conjecture}, we prove Therorem \ref{regular}.
\begin{proof}[Proof of Therorem \ref{regular}]
Let $\mathbf{a}=\{d(1,n)\}_{n\geq0}, \mathbf{b}=\{d(2,n-1)d(3,n-1)\}_{n\geq0}$ and $\mathbf{c}=\{d(2,n)d(3,n-1)\}_{n\geq0}$.
By Proposition \ref{formula2}-\ref{formula3}, we know that $\{d(1,n)\}_{n\geq0}$ is periodic with period $2$ and $\{d(2,n)\}_{n\geq0}$ is $2$-automatic.
Hence, by Theorem \ref{cloure}, we know that $\mathbf{a}, \mathbf{b}$ and $\mathbf{c}$ are $2$-automatic sequences.
By Lemma \ref{formula1}, the sequence $\{d(0,n)\}_{n\geq0}$ is a polynomial generated sequence with shift, as
$$d(0,2n)=a(n)d(0,n)-b(n),~d(0,2n+1)=a(n)d(0,n+1)-c(n).$$
By Corollary \ref{coro-poly2}, the sequene $\{d(0,n)\}_{n\geq0}$ is $2$-regular.

If $m\geq3$ is odd, then by Proposition \ref{formula4} and Theorem \ref{conjecture}, we know that the sequence $\{d(m,n)\}_{n\geq0}$  is periodic. Moreover,  $2^{k+1}$ is a period if $2^k<m\leq 2^{k+1}$.
If $m\geq3$ is even, then by Proposition \ref{formula4} and Theorem \ref{conjecture}, we know that the sequence $\{d(m,n)\}_{n\geq0}$  is $2$-automtic.
We completes this proof.
\end{proof}

\begin{remark}
Note from Proposition \ref{formula2} that $d(m,n)\in\{-1,0,1\}$,  we know that $$d^{k}(m,n)\in\{d(m,n),d^2(m,n)\}$$ for any integer $k\geq0$. By Lemma \ref{formula1}, it is easy to check directly that the $2$-kernel of the sequence $\{d(m,n)\}_{m\geq1,n\geq0}$  is finite and the $2$-kernel of the sequence $\{d(m,n)\}_{m\geq0,n\geq0}$  is finitely generated.

Hence, the two-dimensional sequence $\{d(m,n)\}_{m\geq1,n\geq0}$ is $2$-automatic and $\{d(m,n)\}_{m\geq0,n\geq0}$ is $2$-regular.  An immediate consequence of a result of Salon in \cite{S1986,S1987} is that the sequences $\{d(m,n)\}_{n\geq0}$  are $2$-automatic for all $m\geq1$. More about multidimensional automatic sequences and regular sequences, please see \cite{AS03,jp}.

\end{remark}

\end{document}